\documentclass[12pt,a4paper,fleqn,leqno]{amsart}
\usepackage{amssymb}
\usepackage{mathrsfs}
\usepackage[all]{xy}
\usepackage{fancyhdr}
\usepackage{xcolor}
\usepackage[bookmarks,colorlinks]{hyperref}
\hypersetup{colorlinks=true,%
            citecolor=green,%
            filecolor=black,%
            linkcolor=red,%
            urlcolor=magenta,%
            pdftex}

\theoremstyle{plain}
\newtheorem{theorem}{Theorem}[section]

\newtheorem{corollary}[theorem]{Corollary}
\newtheorem{proposition}[theorem]{Proposition}
\newtheorem{claim}{Claim}

\theoremstyle{definition}

\newtheorem{remark}[theorem]{Remark}
\newtheorem{example}[theorem]{Example}

\renewcommand{\emptyset}{\varnothing}

\DeclareMathOperator{\uhr}{\upharpoonright}

\DeclareMathOperator{\st}{St} 
 
\DeclareMathOperator{\conv}{conv} 

\DeclareMathOperator{\N}{\mathbb{N}}

\DeclareMathOperator{\R}{\mathbb{R}}

\numberwithin{equation}{section}

\begin{document}

\title{Selections, Paraconvexity and PF-normality}

\thanks{The work of the author is based upon research supported by the
College of Agriculture, Engineering and Science of the University of KwaZulu-Natal, South Africa.}

\author{Narcisse Roland Loufouma Makala}

\address{School of Mathematics, Statistics and Computer Science, University of
KwaZulu-Natal,   Westville Campus, Private Bag X54001, Durban 4000, South Africa}

\email{{\color{magenta}roland@aims.ac.za}}

\subjclass[2010]{Primary 54C60, 54C65; Secondary 54D15.}

\keywords{Set-valued mapping, lower semi-continuous, selection,
paraconvexity, PF-normal space.}

\begin{abstract}
We prove a selection theorem for paraconvex-valued mappings defined on $\tau$-PF normal spaces. The method developed to prove this result is used to provide a general approach
to such selection theorems.
\end{abstract}

\date{\today}
\maketitle

\section{Introduction}
\label{section1}

Let $X$ and $Y$ be topological spaces, and let $2^Y$ be the family of all non-empty subsets of $Y$. Also, let
\[
\mathscr{F}(Y)=\{S\in 2^Y:S\ \text{is closed}\}\ \text{and}\ \mathscr{C}(Y)=\{S\in \mathscr{F}(Y):S\ \text{is compact}\}.
\]

A set-valued mapping $\varphi:X\to 2^Y$ is \emph{lower semi-continuous}, or \emph{l.s.c.}, if the set
\[
\varphi^{-1}(U)=\{x\in X: \varphi(x)\cap U\neq \emptyset\}
\]
is open in $X$ for every open $U\subset Y$. A single-valued mapping $f:X\to Y$ is a \emph{selection} for $\varphi:X\to 2^Y$ if $f(x)\in \varphi(x)$ for every $x\in X$.\medskip

Let $Y$ be a normed space. Throughout this paper, we will use $d$ to denote the metric on $Y$ generated by the norm of $Y$. Following \cite{michael2}, a subset $P$ of $Y$ is called
$\alpha$-\emph{paraconvex}, where $0\leq\alpha\leq1$, if whenever $r>0$ and $d(p,P)<r$ for some $p\in Y$, then
\[
d(q,P)\leq\alpha r\ \text{for all}\ q\in\conv(B_r(p)\cap P).
\]

Here, $B_r(x)=\{y\in Y:d(x,y)<r\}$, and $\conv(A)$ is the \emph{convex hull} of $A$. The set $P$ is called \emph{paraconvex} if it is $\alpha$-paraconvex for some $\alpha<1$. A closed set is $0$-paraconvex if and only if it is convex.\medskip

In the sequel, $\tau$ will denote an infinite cardinal number, and $w(Y)$ --- the topological weight of the space $Y$. Also, we will use $\mathscr{C}'(Y)=\mathscr{C}(Y)\cup\{Y\}$. \medskip

In \cite{michael}, E. Michael proved that if $X$ is paracompact and $Y$ is a Banach space, then every l.s.c.\ convex-valued mapping $\varphi:X\to\mathscr{F}(Y)$ has a continuous selection (see \cite[Theorem $3.2^{\prime\prime}$]{michael}). In \cite{michael2}, E. Michael generalized this result by replacing ``convexity" with ``$\alpha$-paraconvexity" for a fixed $\alpha<1$ (see \cite[Theorem 2.1]{michael2}); this generalization remains valid for $\tau$-paracompact normal spaces, see \cite[Theorem 3.2]{loufouma}. P.V. Semenov generalized Michael's paraconvex-valued selection theorem by replacing the constant $\alpha$ with a continuous function $f:(0,\infty)\to[0,1)$ satisfying a certain property called $(PS)$ (\emph{functional paraconvexity}, see \cite{semenov}); and D. Repov\v{s} and P.V. Semenov considered in \cite{repov-semenov} a function $\alpha_P:(0,\infty)\to[0,2)$ (called the \emph{function of nonconvexity}) associated to each nonempty subset $P\subset Y$, see \cite{repov-semenov,semenov} for the definition of these concepts. Also, they obtained several applications of selections for paraconvex-valued mappings, see \cite{repov-semenov2,repov-semenov3,repov-semenov4,repov-semenov5,repov-semenov6,semenov2}
and the monograph \cite{repov-semenov7}. The author has recently proved a $\tau$-collectionwise normal version of these results, i.e. when $X$ is $\tau$-collectionwise normal, $Y$ is a Banach space with $w(Y)\leq\tau$, and $\varphi$ is $\alpha$-paraconvex- and $\mathscr{C}'(Y)$-valued \cite[Theorem 2.1]{loufouma}. Let us explicitly remark that the proofs of these theorems utilize the fact that $\tau$-paracompactness and $\tau$-collectionwise normality are hereditary with respect to closed subsets.\medskip

We are now ready to state also the main purpose of this paper. Namely, we prove a paraconvex-valued selection theorem for $\mathscr{C}(Y)$-valued mappings defined on $\tau$-PF-normal spaces (Corollary \ref{pf-thm0}), see Section \ref{dmp-sel-thm} for the definitions of these spaces. The challenge in this particular case is that $\tau$-PF-normality is not hereditary with respect to closed subsets, hence the method used for the $\tau$-collectionwise normal spaces in \cite{loufouma} cannot be applied straightforward; the rest of the arguments are similar. In fact, using the property of mappings discussed in the next section, we prove a slightly more general result (Theorem \ref{pf-thm}) and derive from a common point of view all previous known results of paraconvex-valued selection theorems for l.s.c.\ mappings (see Examples \ref{example1} and \ref{cwn-exple}).

\section{The Dense Multi-selection Property}
\label{dmp}

For $\varepsilon>0$, a single-valued mapping $g:X\to Y$ to a metric space $(Y,d)$ is an \emph{$\varepsilon$-selection} for $\varphi:X\to 2^Y$, if $d(g(x),\varphi(x))<\varepsilon$, for every $x\in X$. Also, a set-valued mapping $\psi:X\to 2^Y$ is called a \emph{set-valued selection} (or a \emph{multi-selection}) for another set-valued mapping $\varphi:X\to 2^Y$ if $\psi(x) \subset\varphi(x)$, for every $x\in X$. We shall say that a mapping $\varphi:X\to 2^Y$ has the \emph{Dense Multi-selection Property}, or \emph{DMP} for short, where $(Y,d)$ is a metric space, if the following hold:
\begin{enumerate}
\item[$(i)$] $\varphi$ has an l.s.c.\ multi-selection $\psi:X\to\mathscr{C}(Y)$.
\item[$(ii)$] For every $\varepsilon>0$, a cozero-set $U\subset X$, and a continuous $\varepsilon$-selection $g:U\to Y$ for $\varphi\uhr U$, there exists an l.s.c.\ $\psi:U\to\mathscr{C}(Y)$ such that
\[
\psi(x)\subset\overline{\varphi(x)\cap B_{\varepsilon}(g(x))},\ x\in U.
\]
\end{enumerate}

We may consider open balls $B_{\varepsilon}(y)$ when $\varepsilon=\infty$. Thus,
$B_{\infty}(y)=Y$, and the DMP of $\varphi:X\to 2^Y$ can be simply expressed by saying that for every $0<\varepsilon\leq\infty$, a cozero-set $U\subset X$, and a continuous $\varepsilon$-selection $g:U\to Y$ for $\varphi\uhr U$, there exists an l.s.c.\ $\psi:U\to\mathscr{C}(Y)$ such that $\psi(x)\subset\overline{\varphi(x)\cap B_{\varepsilon}(g(x))}$, $x\in U$.

\begin{remark}
\label{remark1}
In the realm of normal spaces, cozero-sets coincide with open $F_{\sigma}$-sets. So, if $X$ is normal, $U\subset X$ is an open $F_{\sigma}$-set, and $\varphi:X\to\mathscr{F}(Y)$ has the DMP, then $\varphi\uhr U$ will automatically have the DMP.
\end{remark}

The following are some examples of mappings satisfying the DMP. In the first example, a space $X$ is \emph{$\tau$-paracompact} if it is Hausdorff and every open cover $\mathscr{U}$ of $X$, with $|\mathscr{U}|\leq\tau$, has a locally finite open refinement.

\begin{example}
\label{example1}
If $X$ is $\tau$-paracompact and normal, $(Y,d)$ is a complete metric space with $w(Y)\leq\tau$, and $\varphi:X\to\mathscr{F}(Y)$ is l.s.c., then $\varphi$ has the DMP.
\end{example}

\begin{proof}
Take a cozero-set $U\subset X$; i.e.\ an open $F_{\sigma}$-subset of X, and a continuous $\varepsilon$-selection $g:U\to Y$ for $\varphi\uhr U$. Note that $U$ is normal (see, for instance, \cite[Problem 2.7.2 (b)]{Engelking}) and it is also $\tau$-paracompact (see \cite[Proposition 3]{michael3}). By \cite[Theorem 1.1]{michael5} (see also \cite[Theorem 11.2]{choban}), the l.s.c.\ mapping $\Phi:U\to\mathscr{F}(Y)$ defined by $\Phi(x)=\overline{\varphi(x)\cap B_{\varepsilon}(g(x))}$, $x\in U$, admits an l.s.c.\ multi-selection $\psi:U\to\mathscr{C}(Y)$; i.e. $\varphi$ has the DMP.
\end{proof}

Note that in the special case of $\tau=\omega$, Example \ref{example1} implies that if $X$ is countably paracompact and normal, $(Y,d)$ is a separable complete metric space, and $\varphi:X\to\mathscr{F}(Y)$ is l.s.c., then $\varphi$ has the DMP.\medskip

A $T_1$-space $X$ is said to be \emph{$\tau$-collectionwise normal} if for every discrete collection $\mathscr{D}$ of closed subsets of $X$, with $|\mathscr{D}|\leq\tau$, there exists a discrete collection $\big\{U_D: D\in \mathscr{D}\big\}$ of open subsets of $X$ such that $D\subset U_D$ for every $D\in \mathscr{D}$. Following Nedev \cite{nedev}, for a normal space $X$ and a metric space $(Y,d)$, a mapping $\varphi:X\to\mathscr{F}(Y)$ has the \emph{Selection Factorization Property} (\emph{SFP}, for short) if for every closed subset $F$ of $X$ and every locally finite collection $\mathscr{U}$ of open subsets of $Y$ such that $\varphi^{-1}(\mathscr{U})=\{\varphi^{-1}(U):U\in\mathscr{U}\}$ covers $F$, there exists a locally finite open (in $F$) cover of $F$ which refines $\varphi^{-1}(\mathscr{U})$.

\begin{example}
\label{cwn-exple}
If $X$ is $\tau$-collectionwise normal, $(Y,d)$ is a complete metric space with $w(Y)\leq\tau$, and $\varphi:X\to\mathscr{C}'(Y)$ is l.s.c., then $\varphi$ has the DMP.
\end{example}

\begin{proof}
Let $U\subset X$ be a cozero-set. Then, $U$ is $\tau$-collectionwise normal as an $F_{\sigma}$-subset of $X$ \cite{sediva} (see also \cite[Problem 5.5.1 (b)]{Engelking}). Let $g:U\to Y$ be a continuous $\varepsilon$-selection for $\varphi\uhr U$. Define a mapping $\Omega:U\to\mathscr{F}(Y)$ by $\Omega(x)=\overline{B_{\varepsilon}(g(x))}$, $x\in U$. Then, $\Omega$ is proximal continuous in the sense of \cite{gutev}. Define another l.s.c.\ mapping $\Phi:U\to\mathscr{F}(Y)$ by $\Phi(x)=\overline{\varphi(x)\cap B_{\varepsilon}(g(x))}$, $x\in U$. Note that $\Phi$ is a multi-selection for $\Omega$, and $\Phi(x)\neq\Omega(x)$ implies that $\Phi(x)$ is compact. By \cite[Lemma 4.2]{gutev-al}, $\Phi$ has the SFP. Hence, by \cite[Proposition 4.1]{nedev}, there exists an l.s.c.\ $\psi:U\to\mathscr{C}(Y)$ such that
\[
\psi(x)\subset\Phi(x)=\overline{\varphi(x)\cap B_{\varepsilon}(g(x))},\ x\in U.
\]
Thus, $\varphi$ has the DMP.
\end{proof}

\section{A selection theorem for paraconvex-valued mappings with $\tau$-PF-normal domain}
\label{dmp-sel-thm}

Let $X$ be a topological space. The \emph{star} of a set $A\subset X$ with respect to a cover $\mathscr{U}$ of $X$ is the set $\st(A,\mathscr{U})=\bigcup\{U\in\mathscr{U}:A\cap U \neq\emptyset\}$. A cover $\mathscr{U}$ of $X$ is said to be a \emph{star-refinement} of
another cover $\mathscr{V}$ of $X$ if for each $U\in\mathscr{U}$, there is some $V\in\mathscr{V}$ such that $\st(U,\mathscr{U})\subset V$. A \emph{normal sequence} in a space $X$ is a sequence $\{\mathscr{U}_n:n\in\N\}$ of open covers of $X$ such that $\mathscr{U}_{n+1}$ is a star-refinement of $\mathscr{U}_n$, for each $n\in\N$. An open cover $\mathscr{U}$ of $X$ is called a \emph{normal cover} if $\mathscr{U}=\mathscr{U}_1$ for some normal sequence $\{\mathscr{U}_n:n\in\N\}$ of open covers of $X$. It is well known that an open cover $\mathscr{U}$ of a space $X$ is normal if and only if it admits a locally finite refinement consisting of cozero-sets \cite{morita}. A $T_1$-space $X$ is \emph{$\tau$-PF-normal} if every point-finite open cover $\mathscr{U}$ of $X$, with $|\mathscr{U}|\leq\tau$, is normal. A space $X$ is \emph{PF-normal} if it is $\tau$-PF-normal for every $\tau$; or if it is normal and every point-finite open cover of $X$ has a locally finite open refinement. Note that in the realm of normal spaces, $\tau$-PF-normal spaces coincide with $\tau$-pointwise-$\aleph_0$-paracompact spaces in Nedev's terminology \cite{nedev}; and PF-normal spaces coincide with point-finitely paracompact spaces in the sense of Kand\^{o} \cite{kando}. Every collectionwise normal
space is PF-normal (see \cite[Theorem 2]{michael4}) and every PF-normal space is obviously normal. However, none of these implications is invertible (see \cite[Examples 1 and 2]{michael4}). In contrast to collectionwise normality and paracompactness, PF-normality is not hereditary with respect to closed subsets (see \cite[p. 506, \S 4]{gutev-al}).
However, it was proved in \cite{yamauchi} that PF-normality is hereditary with respect to open $F_{\sigma}$-subsets (see \cite[Proposition 4.5]{yamauchi}). The PF-normal spaces were investigated in \cite{gutev-al,kando,michael4,smith}.\medskip

Let $\mathscr{F}_{\alpha}(Y)$ (resp.\ $\mathscr{C}_{\alpha}(Y)$) be the set of all $\alpha$-paraconvex members of $\mathscr{F}(Y)$ (resp.\ $\mathscr{C}(Y)$). We are going to prove the following theorem.

\begin{theorem}
\label{pf-thm}For a Banach space $Y$, with $w(Y)\leq \tau$, and $0\leq\alpha<1$, the following hold\textup{:}
\begin{itemize}
\item[(a)] Whenever $X$ is a $\tau$-PF-normal space, every mapping $\varphi:X\to\mathscr{F}_{\alpha}(Y)$ having the DMP has a continuous selection.
\item[(b)] There exists $\delta=\delta(\alpha)>0$ depending only on $\alpha$ such that if $X$ is a $\tau$-PF-normal space, $\varphi:X\to\mathscr{F}_{\alpha}(Y)$ has the DMP, and $g:X\to Y$ is a continuous $r$-selection for $\varphi$ for some $r>0$, then $\varphi$ has a continuous selection $f:X\to Y$ with $d(g(x),f(x))<\delta r$, for all $x\in X$.
\end{itemize}
\end{theorem}

To prepare for the proof of Theorem \ref{pf-thm}, we need the following construction.

\begin{claim}
\label{claim}
Let $X$ be a space, $(Y,d)$ be a metric space, $\varphi:X\to 2^Y$ have the DMP, $A\subset X$ be closed and $U\subset X$ be a neighbourhood of $A$. If $g:U\to Y$ is a continuous selection for $\varphi\uhr U$, then the mapping $\varphi_g:X\to 2^Y$ defined by
\begin{displaymath}
\varphi_{g}(x)=\left\{\begin{array}{ll}
      \{g(x)\} & \textrm{if $x\in A$}\\
      \varphi(x) & \textrm{if $x\in X\setminus A,$}
      \end{array}\right.
\end{displaymath}
also has the DMP.
\end{claim}

\begin{proof}
Let $V\subset X$ be a cozero-set and $f:V\to Y$ be a continuous $\varepsilon$-selection for $\varphi_g\uhr V$. Then, $f$ is also a continuous $\varepsilon$-selection for $\varphi\uhr V$. Since $\varphi$ has the DMP, there exists an l.s.c.\ $\psi:V\to\mathscr{C}(Y)$ such that $\psi(x)\subset\overline{\varphi(x)\cap B_{\varepsilon}(f(x))}$, $x\in V$. Define a mapping $\psi_g:V\to\mathscr{C}(Y)$ by
\begin{displaymath}
\psi_{g}(x)=\left\{\begin{array}{lll}
      \{g(x)\} & \textrm{if $x\in A\cap V$}\\
      \{g(x)\}\cup\psi(x) & \textrm{if $x\in (U\cap V)\setminus A$}\\
      \psi(x) & \textrm{if $V\setminus U$}.
      \end{array}\right.
\end{displaymath}
Then, $\psi_g$ is l.s.c.\ and $\psi_g(x)\subset\overline{\varphi_g(x)\cap B_{\varepsilon}(f(x))}$, $x\in V$. Hence, $\varphi_g$ has the DMP.
\end{proof}

We also need the following proposition which is a $\tau$-PF-normal version of
\cite[Proposition 2.2]{loufouma} in terms of DMP mappings. Recall that a set-valued mapping $\psi:X\to 2^Y$ is \emph{upper-semi continuous}, or \emph{u.s.c.}, if the set $\psi^{\#}(U)= \{x\in X:\psi(x)\subset U\}$ is open in $X$ for every open $U\subset Y$. Equivalently, $\psi$ is u.s.c.\ if $\psi^{-1}(F)$ is closed in $X$ for every closed subset $F\subset Y$.

\begin{proposition}
\label{prop1} Let $X$ be a $\tau$-PF-normal space, $Y$ be a completely metrizable space with $w(Y)\leq\tau$, $\{V_n:n\in\N\}$ be an increasing open cover of $Y$, and $\varphi:X\to\mathscr{F}(Y)$ have the DMP. Then, there exists an increasing closed cover
$\{A_n:n\in\N\}$ of $X$ such that $A_n\subset\varphi^{-1}(V_n)$, for every $n\in\N$.
\end{proposition}

\begin{proof}
Since $\{V_n:n\in\N\}$ is an increasing open cover of $Y$ and $Y$ is normal and countably paracompact (being metrizable), there exists an increasing closed cover $\{F_n:n\in\N\}$ of $Y$ such that $F_n\subset V_n$, for every $n\in\N$. We then have
\[
\varphi^{-1}(F_n)\subset\varphi^{-1}(V_n),\ \text{for every}\ n\in\N.
\]

Since the mapping $\varphi$ has the DMP, by (i) of the definition of DMP, $\varphi$ admits an l.s.c.\ multi-selection $\psi:X\to\mathscr{C}(Y)$. Since $X$ is $\tau$-PF normal, $\psi$ has a u.s.c. multi-selection $\phi:X\to\mathscr{C}(Y)$ (see \cite[Theorem 4.3]{nedev}). We then have
\[
\phi^{-1}(F_n)\subset\psi^{-1}(F_n)\subset\varphi^{-1}(F_n)\subset\varphi^{-1}(V_n),\ \text{for every}\ n\in\N.
\]

The family $\{A_n:n\in\N\}$, with $A_n=\phi^{-1}(F_n)$, is an increasing closed cover of $X$ such that $A_n\subset\varphi^{-1}(V_n)$, for every $n\in\N$.
\end{proof}

We are now ready to prove Theorem \ref{pf-thm}.

\begin{proof}[Proof of Theorem \ref{pf-thm}]
We are going to first prove (b), and then (a).\medskip

(b) Since $\alpha<1$, there exists $\gamma\in\R$ such that $\alpha<\gamma<1$. Then, $\sum^\infty_{i=0}\gamma^i<\infty$ (i.e. the series $\sum^\infty_{i=0}\gamma^i$ converges). So, take $\delta$ such that $\sum_{i=0}^{\infty}\gamma^i< \delta$. This $\delta$ works.
Indeed, let $r>0$ and $g:X\to Y$ be a continuous $r$-selection for $\varphi$. We shall define by induction a sequence of continuous maps $f_n:X\to Y$, $n<\omega$, with $f_0=g$, satisfying the conditions for all $n\in\N$ and $x\in X$:
\begin{eqnarray}
\label{eq1} d(f_n(x),\varphi(x)) & < & \gamma^nr,\\
\label{eq2} d(f_n(x),f_{n+1}(x)) & \leq &\gamma^nr.
\end{eqnarray}

This will be sufficient because by \eqref{eq2}, $\{f_n:n<\omega\}$ is a Cauchy sequence, so it must converge to some continuous map $f:X\to Y$. By \eqref{eq1}, $f(x)\in\varphi(x)$, for every $x\in X$ and, by \eqref{eq2}, $d(g(x),f(x))<\delta r$, $x\in X$.\medskip

Let $f_0=g$, which satisfies \eqref{eq1}. Suppose that $f_n:X\to Y$ has been constructed for some $n\geq0$, and let us construct $f_{n+1}$. By the inductive assumption, $f_n$ is a continuous $\gamma^nr$-selection for $\varphi$. Define a mapping $\Phi_{n+1}:X\to \mathscr{F}(Y)$ by
\[
\Phi_{n+1}(x)=\overline{\varphi(x)\cap B_{\gamma^nr}(f_n(x))},\ x\in X.
\]

Since $\varphi$ has the DMP, the mapping $\Phi_{n+1}$ has an l.s.c.\ multi-selection $\psi_{n+1}:X\to\mathscr{C}(Y)$. Since $X$ is $\tau$-PF normal, by \cite[Theorem 4.1]{nedev}, there exists a continuous map $f_{n+1}:X\to Y$ such that
\[
f_{n+1}(x)\in\overline{\conv(\psi_{n+1}(x))},\ x\in X.
\]

Since $\psi_{n+1}(x)\subset\Phi_{n+1}(x)\subset\overline{\conv(\varphi(x)\cap B_{\gamma^nr}(f_n(x)))}$, we also have that

\[
f_{n+1}(x)\in\overline{\conv(\psi_{n+1}(x))}\subset\overline{\conv(\varphi(x)\cap B_{\gamma^nr}(f_n(x)))},\ x\in X.
\]

By $\alpha$-paraconvexity of $\varphi(x)$, we get that
\[
d(f_{n+1}(x),\varphi(x))\leq\alpha\gamma^n r<\gamma^{n+1}r,\ \text{for all}\ x\in X,
\]
which is \eqref{eq1}. Clearly, we also have
\[
d(f_n(x),f_{n+1}(x))\leq \gamma^nr,\ \text{for all}\ x\in X,
\]
which is \eqref{eq2}.\medskip

(a) Take $\lambda\geq2$ such that $\varphi(x)\cap B_{\lambda}(0)\neq\emptyset$ for some $x\in X$, where $0$ is the origin of $Y$, and let $\beta=\max\{\delta,\lambda\}$, where $\delta$ is as in (b) applied to $\alpha$. Let
\[
V_n=B_{\beta^{n+1}}(0),\ \text{for each}\ n\in\N.
\]

Then, the family $\{V_n:n\in\N\}$ is an increasing open cover of $Y$. By Proposition \ref{prop1}, there is an increasing closed cover $\{A_n:n\in\N\}$ of $X$ such that
$A_n\subset\varphi^{-1}(V_n)$, for every $n\in\N$. Since $X$ is normal, there is an increasing family $\{G_n:n\in\N\}$ of open sets such that $A_n\subset G_n\subset \overline{G}_n\subset\varphi^{-1}(V_n)$, $n\in\N$. Since $\overline{G}_n\subset \varphi^{-1}(V_n)$ for every $n\in\N$, there exists an open $F_{\sigma}$-set $U_n\subset X$ such that $\overline{G}_n\subset U_n\subset\varphi^{-1}(V_n)$. We are going to construct by induction partial selections $g_n:U_n\to V_n$ for $\varphi_n=\varphi\uhr U_n$ such that $g_{n+1}\uhr\overline{G}_n=g_n\uhr\overline{G}_n$, $n\in\N$. To this end, let $g(x)=0$ for all $x$, and take $r$ to be $\beta$. Since $\varphi_1=\varphi\uhr U_1$ has the DMP (see Remark \ref{remark1}), it follows from (b) that it has a continuous selection $g_1:U_1\to V_1$. Suppose that $g_n:U_n\to V_n$ is a continuous selection for $\varphi_n:U_n\to \mathscr{F}_{\alpha}(Y)$. Next, define a mapping $\psi_{n+1}:U_{n+1}\to \mathscr{F}_{\alpha}(Y)$ by
\begin{displaymath}
\psi_{n+1}(x)=\left\{\begin{array}{ll}
      \{g_n(x)\} & \textrm{if $x\in\overline{G}_n$}\\
      \varphi(x) & \textrm{if $x\in U_{n+1}\setminus\overline{G}_n.$}
      \end{array}\right.
\end{displaymath}
By Claim \ref{claim}, $\psi_{n+1}$ has the DMP, and it follows from (b), with $g(x)=0$ for all $x$, and substituting $r$ with $\beta^{n+1}$, that $\psi_{n+1}$ has a continuous selection $g_{n+1}:U_{n+1}\to V_{n+1}$. In particular, $g_{n+1}$ is a continuous selection for $\varphi_{n+1}$ and $g_{n+1}\uhr\overline{G}_n=g_n\uhr\overline{G}_n$. This completes the construction of the partial selections $g_n$, $n\in\N$. Finally, define
$f_n=g_n\uhr\overline{G}_n$, $n\in\N$. Then, $f_n$ is a continuous selection for $\varphi\uhr\overline{G}_n$ such that $f_{n+1}\uhr\overline{G}_n=f_n$, $n\in\N$. This allows us to define a map $f:X\to Y$ by $f\uhr\overline{G}_n=f_n$, and this $f$ is a continuous
selection for $\varphi$. The proof is completed.
\end{proof}

It is worth mentioning that Theorem \ref{pf-thm}, together with Examples \ref{example1} and \ref{cwn-exple}, implies the $\tau$-paracompact normal and $\tau$-collectionwise normal versions of the paraconvex-valued selection theorems for l.s.c.\ mappings; i.e. \cite[Theorem 3.2]{loufouma} and \cite[Theorem 2.1]{loufouma} respectively. It also implies the following result which is a $\tau$-PF-normal version of theorems \cite[Theorem 2.1]{loufouma} and \cite[Theorem 2.1]{michael2}; and is a generalization of \cite[Theorem IV]{kando} (see also \cite[Theorem 4.1]{nedev} in the special case of compact-valued mappings) in terms of paraconvex-valued mappings defined on $\tau$-PF-normal spaces.

\begin{corollary}
\label{pf-thm0}For a Banach space $Y$, with $w(Y)\leq \tau$, and $0\leq\alpha<1$, the following hold\textup{:}
\begin{itemize}
\item[(a)] Whenever $X$ is a $\tau$-PF-normal space, every l.s.c.\ $\varphi:X\to \mathscr{C}_{\alpha}(Y)$ has a continuous selection.
\item[(b)] There exists $\delta=\delta(\alpha)>0$ depending only on $\alpha$ such that if $X$ is a $\tau$-PF-normal space, $\varphi:X\to\mathscr{C}_{\alpha}(Y)$ is l.s.c., and $g:X\to Y$ is a continuous $r$-selection for $\varphi$ for some $r>0$, then $\varphi$ has a continuous selection $f:X\to Y$ with $d(g(x),f(x))<\delta r$, for all $x\in X$.
\end{itemize}
\end{corollary}

\begin{proof}
It is enough to show that $\varphi$ has the DMP, and the result will follow from Theorem \ref{pf-thm}. The proof of this fact is similar to that of Example \ref{example1}, but now using \cite[Theorem 4.3]{nedev} instead of \cite[Theorem 1.1]{michael5}.
\end{proof}

\begin{remark}

As Theorem \ref{pf-thm}, Examples \ref{example1}, \ref{cwn-exple} and Corollary  \ref{pf-thm0} show, the DMP plays the role of a unified approach to ``paraconvex-valued" selection theorems. The Selection Factorization Property due to S.\ Nedev \cite{nedev} plays a similar role for ``convex-valued" selection theorems. Thus, as the referee remarked to the author, it is natural to ask whether every mapping having the SFP has also the DMP; or more interestingly, if Theorem \ref{pf-thm} will remain true if one replaces DMP with SFP. Regarding this, it is to be noted that Proposition \ref{prop1} remains true if DMP is replaced by SFP. Also, the statement in Remark \ref{remark1} is true if DMP is replaced by SFP; i.e.\ if $X$ is a normal space, $U\subset X$ is an open $F_{\sigma}$-set, and $\varphi:X\to\mathscr{F}(Y)$ has the SFP, then $\varphi\uhr U$ also has the SFP. However, it is not evident that a mapping having the SFP satisfies the condition $(b)$ of the definition of DMP, i.e., the following question is open: if $X$ is normal, $\varphi:X\to\mathscr{F}(Y)$ has the SFP, $\varepsilon>0$, $U\subset X$ is a cozero-set, $g:U\to Y$ is a continuous $\varepsilon$-selection for $\varphi\uhr U$. Does the mapping $\Phi:U\to\mathscr{F}(Y)$ defined by $\Phi(x)= \overline{\varphi(x)\cap B_{\varepsilon}(g(x))}$, $x\in U$, admit an l.s.c.\ multi-selection $\psi:U\to \mathscr{C}(Y)$?
\end{remark}

\section{Acknowledgment}

The author would like to express his deep gratitude to Professor V. Gutev for guiding him
in the preparation of this paper. The author is also grateful to the referee for his useful comments and suggestions.
\vspace{1cm}

\newcommand{\noopsort}[1]{} \newcommand{\singleletter}[1]{#1}
\providecommand{\bysame}{\leavevmode\hbox to3em{\hrulefill}\thinspace}
\providecommand{\MR}{\relax\ifhmode\unskip\space\fi MR }
\providecommand{\MRhref}[2]{%
  \href{http://www.ams.org/mathscinet-getitem?mr=#1}{#2}
}
\providecommand{\href}[2]{#2}

\end{document}